\documentclass[reqno]{amsart}
\usepackage{amsfonts,amsmath,amsthm,amssymb}
\usepackage[normalem]{ulem}
\usepackage{enumerate,nicefrac}
\usepackage{hyperref,multicol}
\usepackage{rotating}
\usepackage[all]{xy}
\usepackage{tikz}
\usetikzlibrary{arrows, decorations.markings, shapes.geometric}

\tikzstyle{steps} = [rectangle, rounded corners, text width=3cm, minimum height=1.5cm,text centered, draw=black,ultra thick]
\tikzstyle{mid} = [draw=none,text centered]
\tikzstyle{endbox} = [ellipse, minimum width=3cm, text centered, draw=black,very thick]

\usepackage{mathtools}

\newtheorem{theorem}{Theorem}[section]
\newtheorem{lemma}[theorem]{Lemma}
\newtheorem{proposition}[theorem]{Proposition}

\theoremstyle{definition}

\newtheorem{example}[theorem]{Example}

\usepackage[ansinew]{inputenc}
\usepackage{ifthen}
\usepackage{animate}

\newcommand{\be}{\begin{equation}}
\newcommand{\ee}{\end{equation}}

\title[Integers and their powers in two unrelated number systems]{Represent{i}ng an integer and its powers\\ in two unrelated number systems}

\author{Divyum Sharma}
\address{Department of Mathematics\\
Birla~Inst{i}tute~of~Technology~and~Science, Pilani 333\,031 \textsc{India}}
\email{divyum.sharma\symbol{64}pilani.bits-pilani.ac.in}
\author{L.~Singhal}
\address{Apni Manzil, Nai Sadak, Sajjanon ki Dhani, Kajara 333\,030 \textsc{India}}
\email{lovysinghal\symbol{64}gmail.com}

\begin{document}

\date{}
\subjclass[2020]{11D61, 11J86, 11A63, 11B39}
\keywords{Exponential Diophantine equations, Baker's method, Linear forms in logarithms}

 \begin{abstract}
Let $\alpha$ be a f{i}xed quadrat{i}c irrat{i}onal. Consider the equation 
\[
      y^a\ =\ q_{N_1}+\cdots+q_{N_K},\quad N_1 \geq \cdots \geq N_{K} \geq 0,\quad a, y \geq 2,
\]
where $(q_N)_{N\,\geq\,0}$ is the sequence of convergent denominators to $\alpha$. We find two effective upper bounds for $y^a$ which depend on the Hamming weights of $y$ with respect to its radix and Zeckendorf representations, respect{i}vely. The lat{t}er bound extends a recent result of Vukusic and Ziegler. En route, we obtain an analogue of a theorem by Kebli, Kihel, Larone and Luca.
 \end{abstract}

\maketitle

\section{Introduction}\label{sec_intro}
  \noindent The Fibonacci numbers are defined as $F_0=0,\ F_1 = 1$ and
  \[
    F_{k+2} = F_{k+1} + F_k \text{\ for\ } k \geq 0.
  \]
  It was a long-standing conjecture that $0, 1, 8$ and $144$ are the only Fibonacci Numbers which are perfect powers. This was confirmed by Bugeaud, Mignot{t}e and Siksek using Baker's theory of linear forms in logarithms and the modular method~\cite{BMS_06}. Luca and Patel~\cite{LP18} considered the problem of enumerating perfect powers amongst sums of two Fibonacci numbers. They conjectured that the largest such power is $3864^2=F_{36}+F_{12}$.  More precisely, it was surmised that solutions to
  \begin{equation}\label{E:conj}
    y^a = F_n + F_m,\quad y\geq 1,\ a\geq 2,\ n-2\geq m\geq 2
  \end{equation}
  in terms of $(y,a,n,m)$ all satisfy $n\leq 36$. Indeed, they proved the conjecture if $n\equiv m\pmod{2}$. The case of opposite parity modulo $2$ remains open.\\[-0.1cm]
  
  \noindent In this direct{i}on, Kebli, Kihel, Larone and Luca~\cite{KKLL_21,KL_21} provided explicit upper bounds for $n,m$ and $a$ in terms of $y$. 
  It was also shown in~\cite{KKLL_21} that the $abc$--conjecture implies that~\eqref{E:conj} has only finitely many solutions. We refer to \cite{AHR_21, BRP_22, BL_16, EEKT_24} for some recent results on this theme.\\[-0.1cm]

  \noindent An important feature of the F{i}bonacci sequence is that it coincides with the sequence of convergent denominators to the quadrat{i}c irrat{i}onal $[ 0; \overline{1} ]$. Our f{i}rst main theorem seems to support a principle enunciated by Bugeaud, Cipu and Mignot{t}e~\cite{BCM_13} that ``in two unrelated number systems, two miracles cannot happen simultaneously for large integers.''
     \begin{theorem}\label{Th:Ham2}
    F{i}x integers $K \geq 1,\ \ell \geq 2$ and $b \geq 2$. Let $\alpha$ be a real quadratic irrat{i}onal whose simple continued fraction expansion is given by
    \[
      \alpha=[a_0; a_1, \ldots, a_{r-1}, \overline{b_0, \ldots, b_{s-1}}],\ r \geq 0, s \geq 1.
    \]
   Denote $(q_N)_{N\,\geq\,0}$ to be the sequence of convergent denominators to $\alpha$. There exists an ef{f}ect{i}vely computable $C$ depending only on $\alpha,K,\ell$ and $b$
    such that every solut{i}on
    \[
      ( y, a, m_1, \ldots, m_{\ell}, N_1, \ldots, N_K),\ y, a \geq 2
    \]
    of simultaneous equat{i}ons
    \begin{align}
      y^a\ &=\ q_{N_1}+\cdots+q_{N_K},\quad N_1 \geq \cdots \geq N_{K} \geq 0, \text{ and}\tag{A}\label{E:ya}\\
      y\ &=\ D_1b^{m_1} + \cdots + D_{\ell}b^{m_{\ell}},\quad m_1 > \cdots > m_{\ell} \geq 0,\quad0 < D_i < b\label{E:y2}\tag{B}
    \end{align}
    sat{i}sf{i}es $y^a \leq C$.
    \end{theorem}
    \noindent Ziegler~\cite{Zi_23} proved that for any fixed $y$, the system~\eqref{E:conj} has at most one solution with $a \geq 1$ unless $y=2,3,4,6$ or $10$. Further, there are no solutions to~\eqref{E:conj} when $a \geq 2$ and $y$ is a sum of two Fibonacci numbers  except for $y= 2,3,4,6$ or $10$. Later, Vukusic and Ziegler~\cite{VZ_24} allowed arbitrary $y$ with bounded Hamming weight with respect to the \textit{Zeckendorf representation} (see \S\,\ref{sec_prelim} for def{i}ni{t}ions)
  and found an explicit upper bound for $y^a$ which depends only on the Hamming weight of $y$. More precisely, they proved
  \begin{theorem}
      Let $\epsilon>0$.  There exists an effectively computable constant $C(\epsilon)$ such that if $y,a,m_1,\ldots,m_{\ell},n,m$ are non-negative integers with $a\geq 2$ for which
      \begin{align*}
           y^a\ &=\ F_n+F_m \text{ and}\\
      y\ &=\ F_{m_1} + \cdots + F_{m_{\ell}},
      \end{align*}
    then
        \begin{equation}\label{E:Ziegler_Bound}
            y^a\ \leq\ \exp(C(\epsilon)\cdot \ell^{(3+\epsilon)\ell^2}).
        \end{equation}
  \end{theorem}
  \noindent We aim to extend their results along two direct{i}ons simultaneously:
  \begin{theorem}\label{Th:Ham}
    F{i}x integers $K \geq 1$ and $\ell \geq 2$. Let $\alpha$ be a real quadratic irrat{i}onal whose simple continued fraction expansion is given by
    \[
      \alpha=[a_0; a_1, \ldots, a_{r-1}, \overline{b_0, \ldots, b_{s-1}}],\ r \geq 0, s \geq 1
    \]
    and such that $\mathbb{Q} ( \alpha ) \neq \mathbb{Q} ( \sqrt{5} )$. Denote $(q_N)_{N\,\geq\,0}$ to be the sequence of convergent denominators to $\alpha$. There exists an effectively computable constant $C'$ depending only on $\alpha,K$ and $\ell$ such that every solut{i}on
    \[
      ( y, a, m_1, \ldots, m_{\ell}, N_1, \ldots, N_K),\ y,a\geq 2
    \]
    of simultaneous equat{i}ons
    \begin{align}
      y^a &= q_{N_1}+\cdots+q_{N_K},\quad N_1 \geq \cdots \geq N_{K} \geq 0, \text{ and}\notag\\
      y &= F_{m_1} + \cdots + F_{m_{\ell}},\quad m_1 \geq \cdots \geq m_{\ell} \geq 1\label{E:y}\tag{B\textquotesingle}
    \end{align}
    sat{i}sf{i}es $y^a \leq C'$.
  \end{theorem}
  \noindent If so desired, the upper bounds $C$ and $C'$ ment{i}oned above can be wri{t}ten explicitly following our proof. However, it will become evident in \S\,\ref{sec_thm1} that these will be very large for us and far from being sharp. For $K=2$, one can follow the arguments of \cite[\S\,7]{VZ_24} to get bounds similar to~\eqref{E:Ziegler_Bound}. In proving Theorem~\ref{Th:Ham}, we will need a generalizat{i}on of~\cite[Theorem~1]{KKLL_21} which is also of independent interest. Denote
  \[
    \log^+\!x\ =\ \log \max \{ x, 3 \}.
  \]
  \begin{theorem}\label{Th:y}
    Let $\alpha$ be as in Theorem~\ref{Th:Ham} and $( q_N )_{N\,\geq\,0}$ be the sequence of its convergent denominators. Further, let $y,a, N_1, \ldots, N_K$ be integers such that $y^a = q_{N_1} + \cdots + q_{N_K},\ N_1 \geq \cdots \geq N_{K} \geq 0$ and $a, y \geq 2$. Then,
        \[
        a \ll N_1 \ll (\log^+\!y)^{K}(\log^+ \log^+\!y)^{K}
        \]
      where both of the implied mult{i}plicat{i}ve constants above depend on $\alpha, K$ and are ef{f}ect{i}vely computable.
    \end{theorem}
    \noindent Consequently, there are only finitely many powers $y^a$ which can be wri{t}ten as a sum of $K$ terms from the sequence $( q_N )_{N\,\geq\,0}$ once we have f{i}xed  integers $y\geq 2$ and $K\geq 1$. These powers can also be effectively listed in principle.\\[-0.1cm]

    \noindent We note that for every real quadratic irrational, the sequence of its convergent denominators satisfies a linear recurrence relation. However, the corresponding characteristic polynomial may not have a dominating root. We overcome this technical di{f}f{i}culty with the help of an idea from Peth\"{o}~\cite{Petho_81} (see also Lenstra and Shallit~\cite{LeSh93}). The idea is to split the sequence $( q_N )$ into finitely many subsequences, each satisfying the same binary recurrence relation albeit with possibly di{f}ferent ini{t}ial terms (see \S\,\ref{sec_prelim} and also~\cite{MS24}). Af{t}erwards, Baker-type estimates for linear forms in logarithms have been used to establish the above theorems. To prove Theorem \ref{Th:Ham}, we use elimination of unknown logarithms and an induction argument building upon~\cite{VZ_24,Zi_19}. In Sect{i}on~\ref{sec_prelim}, we record some preliminary results. Theorem~\ref{Th:y} is dealt with in Sect{i}on~\ref{sec_thm2} and a complete proof of Theorem~\ref{Th:Ham} is given in Sect{i}on~\ref{sec_thm1}. We conclude with an abridged proof of Theorem~\ref{Th:Ham2} in the last sect{i}on.\\[-0.1cm]

    \noindent Throughout the paper, $c_1,c_2,\ldots$ denote effectively computable positive numbers depending only on given $\alpha,K,\ell$ and $b$.

\section{Preliminaries}\label{sec_prelim}
  \noindent This sect{i}on will serve as our arsenal where we collect various techniques and tools which may be useful later.
  \subsection{Numerat{i}on systems}\label{SubS:numsys} Consider any real irrat{i}onal number $\alpha$. In 1922, Ostrowski~\cite{Os_1922} proved that the sequence of convergent denominators with respect to its simple continued fraction expansion forms a \textit{basis} for a numeration system. More precisely, he established that
\begin{theorem}[Ostrowski]
  Let $\alpha = [ a_0; a_1,\ldots]$ and $( q_i )_{i\,\geq\,0}$ be the corresponding sequence of convergent denominators. Then, every non-negative integer $n$ can be expressed uniquely as
  \begin{equation}\label{decomp}
	n=\sum\limits_{0\,\leq\,i\,\leq\,l} \epsilon_{i} q_i,
  \end{equation}
  where integers $\epsilon_{i}$'s sat{i}sfy
	\begin{enumerate}[i)]
	\item $0\leq \epsilon_{0}<a_1$,\label{I:init}
	\item $0\leq \epsilon_{i}\leq a_{i+1}$ for $i\geq 1$, and\label{I:ub}
	\item for all $i\geq 1$, $\epsilon_{i} = a_{i+1}$ implies that  $\epsilon_{i-1}=0$.\label{I:full}
	\end{enumerate}
  \end{theorem}
  \noindent Observe that Condition~\eqref{I:full} ensures that the recursive relation
  \[
    q_{i+1}=a_{i+1}q_{i}+q_{i-1}
  \]
  cannot be used to replace a linear combinat{i}on of summands with another summand. In fact, the three conditions above are equivalent to saying that
 \begin{equation}\label{greediness}
     \sum_{i=0}^{j}\epsilon_{i} q_{i}< q_{j+1}\quad\text{for all } 0 \leq j \leq l.
 \end{equation} 
 We refer the reader to~\cite[Theorem 3.9.1]{AS} for a proof. The expression~\eqref{decomp} has been thereaf{t}er called the \textit{Ostrowski $\alpha$-representation} of $n$. When $\alpha$ is the golden rat{i}o, the sequence of convergent denominators coincides with the sequence of Fibonacci numbers and the Ostrowski $\alpha$-numeration system is known as the \textit{Zeckendorf numeration system} \cite{Ze_72}.\\[-0.1cm]

 \noindent A major dist{i}nct{i}on between Ostrowski $\alpha$-numerat{i}on with respect to an arbitrary $\alpha$ and the Zeckendorf numerat{i}on is with regards to the minimality property. Given any representat{i}on $F_{j_1} + \cdots + F_{j_{\ell}}$ of some non-negat{i}ve integer $n$ where $F_{j}'s$ need not be dist{i}nct and may appear along with either or both of their neighbouring members in the F{i}bonacci sequence, we can write $n$ as a Zeckendorf sum with at most $\ell$ terms~\cite{CHHM+}. Af{t}er replacing length by \textit{sum of digits}, the analogous statement need not be true for general Ostrowski numerat{i}on systems.
 \begin{example}
   Let $\alpha = [ 0; 3, 1, *, *, *, \ldots ]$ so that the sequence of convergent denominators looks like $1 = q_0, 3, 4, \ldots$ We may then express
   \[
     6 = 2 \times q_1
   \]
   so that the corresponding `digital sum' is $2$. As the coef{f}icient of convergent denominator $q_1$ is more than the next part{i}al denominator $a_2=1$, this is not an Ostrowski $\alpha$-representat{i}on of $6$. The same is actually given by $1 \times q_2 + 2 \times q_0$ with sum-of-digits being $1 + 2 = 3$.
 \end{example}
 \noindent It, therefore, becomes prudent for us to consider more general sums of convergent denominators in~\eqref{E:ya} rather than Ostrowski sums when $\alpha \neq ( 1 + \sqrt{5} ) / 2$. On the other hand, we may restrict ourselves to the Zeckendorf decomposi{t}ion of $y$ in~\eqref{E:y} without any loss of generality. This in part{i}cular implies that we can henceforth assume
 \begin{equation}
   m_{j}\ \geq\ m_{j + 1} + 2 \text{ for all } 0 \leq j < \ell.
 \end{equation}
 In addi{t}ion, we may part{i}t{i}on the f{i}nite sequence $N_1, \ldots, N_K$ as
 \[
   N_1 = \cdots = N_{d_1} > N_{d_1 + 1} = \cdots = N_{d_1 + d_2} > \cdots > N_{d_{k - 1} + 1} = \ldots = N_K
 \]
 and collect coe{f}f{i}cients of same convergent denominators to rewrite~\eqref{E:ya} in the form
 \begin{equation}\label{E:yaf9}
   y^a\ =\ d_1q_{N'_1} + \cdots + d_kq_{N'_k},\quad N'_1 > N'_2 > \ldots > N'_k \geq 0.
 \end{equation}
 Note that $d_j \geq 1$ for all $1 \leq j \leq k$ and $\sum_{j = 1}^k d_j = K$ whereby $k \leq K$.
 
  \subsection{Quadrat{i}c irrat{i}onals} Let $\alpha$ be a real quadrat{i}c irrat{i}onal. Lagrange's theorem tells us that its simple cont{i}nued fract{i}on can be wri{t}ten as an eventually periodic sequence
  \[
    \alpha = [ a_0; a_1, \ldots,a_{r-1},\overline{b_0,\ldots,b_{s-1}}],\quad r\geq 0,\ s\geq 1.
  \]
  By~\cite[pg.~352]{LeSh93} (see also \cite[Lemma 2]{Petho_81}), the sequence $(q_{i})_{i\geq 0}$ of convergent denominators to $\alpha$ sat{i}sf{i}es the linear recursive relat{i}on
    \begin{equation}\label{rr_qn}
    q_{i+2s} = t_{\alpha}q_{i+s} - (-1)^sq_{i}\quad\text{ for all } i\geq r,  
    \end{equation}
    where 
    \begin{equation}\label{t_alpha}
    t_{\alpha}=\textrm{trace} \left(\prod_{0\leq j<s}\begin{pmatrix}
    b_j & 1\\
    1 & 0
    \end{pmatrix}\right).
    \end{equation} 
The next lemma gives a Binet-type formula for such convergent denominators.
\begin{lemma}\label{Binet-type}
   Let $\alpha$ be a real quadrat{i}c irrat{i}onal with simple cont{i}nued fract{i}on expansion  $[ a_0; a_1, \ldots,a_{r-1},\overline{b_0,\ldots,b_{s-1}}],\ r \geq 0, s \geq 1$. For $j=0,\ldots,s-1$, we re-index as $q^{(j)}_{i} := q_{j + r + si},\ i\geq 0$. 
   \begin{enumerate}[1)]
       \item Each of the result{i}ng subsequences $\big( q_{i}^{(j)} \big)_{i\,\geq\,0}$ sat{i}sfy\label{I:sub}
       \begin{equation}\label{E:binet_rep}
         q^{(j)}_{i}=c_{1}^{(j)}\theta_{1}^{i}-c_{2}^{(j)}\theta_{2}^i\quad\text{for all } i \geq 0,
       \end{equation}
       where 
    \begin{align*}
        \theta_{1}&=\frac{t_{\alpha}+\sqrt{t_{\alpha}^2-4(-1)^s}}{2},\  \theta_{2} = \frac{t_{\alpha}-\sqrt{t_{\alpha}^2-4(-1)^s}}{2},\\
        c^{(j)}_{1}&=\frac{q^{(j)}_{1}-\theta_{2}q^{(j)}_{0}}{\theta_{1}-\theta_{2}}\quad\!\text{and}\quad\!c_{2}^{(j)} = \frac{q^{(j)}_{1}-\theta_{1}q^{(j)}_{0}}{\theta_{1}-\theta_{2}}.
    \end{align*}
       \item There exist positive constants $c_{3}, c_{4}$ and $N_{0}$ depending only on $\alpha$ such that for every $j = 0, \ldots, s - 1$,
       \begin{equation}\label{qn_ub_lb}
         q^{(j)}_{i} \leq c_{3} \theta_{1}^{i}\ \textrm{ for }\ i \geq 0\ \textrm{ and }\ q^{(j)}_{i} \geq c_{4} \theta_{1}^{i} \textrm{ for }\ i \geq N_{0}.
       \end{equation}
    \end{enumerate}
  \end{lemma}
  \begin{proof}
    From~\eqref{rr_qn} and the definition of $q^{(j)}_{i}$, we have
      \begin{equation}\label{E:rec}
        q^{(j)}_{i+2}-t_{\alpha}q^{(j)}_{i+1}+(-1)^sq^{(j)}_{i}=0,\ i\geq 0.
      \end{equation}
    It is then easy to see that \eqref{I:sub} follows from elementary propert{i}es of binary recursive relat{i}ons. Insofar as our second claim is concerned,
    \[
    t_{\alpha} - 2\ \leq\ \sqrt{t_{\alpha}^2-4(-1)^s}\ \leq\ t_{\alpha}+2,
    \]
    whence
    \[
    \max\{t_{\alpha}-1,1\}< \theta_{1}< t_{\alpha}+1\quad\text{and } -1 < \theta_{2} < 1.
    \]
    The lef{t}most inequality is trivial for $t \geq 2$ while $t_{\alpha}=1$ only when the period $s$ in the cont{i}nued fract{i}on of $\alpha$ equals one (and $\theta_{1} = \nicefrac{(1+\sqrt{5})}{2}$). We alert the reader that both $\theta_{1}$ and $\theta_{2}$ are necessarily irrat{i}onal.
    Furthermore, $c^{(j)}_{1}>0$ for all $j=0,\ldots,s-1$ being coe{f}f{i}cients of respect{i}ve dominat{i}ng terms in~\eqref{E:binet_rep}. One thereby gets for $0 \leq j < s$,
    \[
      q^{(j)}_{i}\leq c_{3}\theta_{1}^i\quad\forall  i \geq 0,\quad\text{where } c_{3} := \max_{0\leq j\leq s-1} \big\{\,c^{(j)}_{1} + \big|c^{(j)}_{2}\big|\,\big\}.
    \]
    We may similarly have a lower bound for the growth of convgerent denominator subsequences. Since $| \theta_{2} / \theta_{1} | < 1$, there exists $N_0 \in \mathbb{N}$ depending only on $\alpha$ such that for all $i\geq N_0$,
    \[
    \left|\frac{\theta_{2}}{\theta_{1}}\right|^i\ <\ \frac{c^{(j)}_{1}}{2\big| c^{(j)}_{2} \big|},\quad j = 0, \ldots, s - 1.
    \]
    Thereaf{t}er, one obtains using triangle inequality that
    \[
      q^{(j)}_{i}\ \geq\ \theta_{1}^i \left| c^{(j)}_{1} - |c^{(j)}_{2}|\left|\frac{\theta_{2}}{\theta_{1}}\right|^i\right|\ >\ \frac{c^{(j)}_{1}}{2}\theta_{1}^i\ \geq\ c_{4}\theta_{1}^i,
    \]
    where $c_{4} := \min_{\,0\,\leq\,j\,\leq\,s-1} c^{(j)}_{ 1} / 2$. Note that both $c_{3}$ and $c_{4}$ depend only on $\alpha$ by our construct{i}on.
  \end{proof}

  \subsection{Heights} Let $\delta$ be an algebraic  number with minimal polynomial
  \[
    f(X) = d_0(X-\delta^{(1)})\cdots(X-\delta^{(d)})\ \in\ \mathbb{Z}[X].
  \]
  The \textit{absolute logarithmic height} $h ( \delta )$ is given by
    \[
      \frac{1}{d} \big( \log d_0+\sum_{i=1}^d\max\{0,\log|\delta^{(i)}|\} \big).
    \]
    For any algebraic numbers $\delta_1, \delta_2$ and rat{i}onal integer $k$, we then have
    \begin{equation}\label{height_of_product}
        h(\delta_1\delta_2^{\pm 1})\ \leq\ h(\delta_1)+h(\delta_2)\quad\text{and}\quad h( \delta_1^k )\ =\ |k| \cdot h (\delta_1)
    \end{equation}
   (see~\cite[Property 3.3]{Wa_00}). The next result provides us an upper bound for absolute logarithmic heights of values taken by integer polynomials at an algebraic number.
\begin{proposition}[\protect{cf.~\cite[Lemma 3.7]{Wa_00}}]\label{height_of_poly}
    Let $f \in \mathbb{Z}[x_1,\ldots,x_T]$ be a non-zero polynomial and $\delta_1,\ldots,\delta_T$ be algebraic numbers. Then,
    \[
    h(f(\delta_1,\ldots,\delta_T))\ \leq\ \sum_{i=1}^T(\deg_{x_i} f)h(\delta_i)+\log L(f)
    \]
    where $L(f)$ denotes the sum of absolute values of coef{f}icients of $f$.
\end{proposition}

  \subsection{Baker-type est{i}mates} We will use the following ref{i}nement due to Matveev \cite{Mat1, Mat2}  for lower bounds on linear forms in logarithms:
\begin{proposition}\label{lfl}
    F{i}x $T$ to be some positive integer. Let  $\delta_1,\ldots,\delta_T$ be positive algebraic numbers and $\log\delta_1,\ldots,\log\delta_T$ be some determinations of their complex logarithms. Further, let $D$ be the degree of the number field generated by $\delta_1,\ldots,\delta_T$ over $\mathbb{Q}$.  For $j=1,\ldots,T$, assume that $A_j \in \mathbb{R}$ sat{i}sfy
    \[
      A_j\ \geq\ \max \big\{ Dh(\delta_j),\,|\log\delta_j|,\,0.16\,\big\}.
    \]
    For $k_1, \ldots, k_T \in \mathbb{Z}$, we set $B = \max \big\{\,|k_1|,\,\ldots,\,|k_T|\,\big\}$,
    \[
      \Gamma\ :=\ \delta_1^{k_1}\cdots\delta_T^{k_T}\quad\text{and}\quad\Lambda\ :=\ \sum_{i=1}^Tk_i\log\delta_i.
    \]
    If $\Gamma\neq 1$, then
    \[
      \log|\Gamma-1|\ >\ -1.4\times 30^{T+3} (T+1)^{4.5} D^{2} \log(eD) A_1 \cdots A_T \log(eB)
    \]
    and
    \[
     \log|\Lambda|\ >\ -2\times 30^{T+4} (T+1)^{6} D^{2} \log(eD) A_1 \cdots A_T \log(eB).
    \]

\end{proposition}
  \subsection{Some elementary inequali{t}ies}\label{SubS:EI}
    Let \eqref{E:ya} and hence,~\eqref{E:yaf9} hold. For $1\leq i\leq k$, we write $N_i'=sn_i+j_i+r$ where each $j_i$ is constrained as $0 \leq j_i\leq s-1$. In the notat{i}on of Lemma~\ref{Binet-type}, one gets
   \begin{equation}\label{E:ya_modified}
     y^a=  \sum_{i=1}^k d_iq^{(j_i)}_{n_i},\quad n_1 \geq \cdots \geq n_{k} \geq 0.
    \end{equation}
    The upper bound in~\eqref{qn_ub_lb} helps us to obtain
    \[
      2^a\ \leq\ y^a\ \leq\ Kq^{(j_1)}_{n_1}\ \leq\ Kc_{3}\theta_{1}^{n_1}.
    \]
    Thereby,
    \begin{equation}\label{E:a_n1}
      a\log y\ \leq\ c_{5}n_1\quad\textrm{and}\quad a\ \leq\ c_{6}n_1
    \end{equation}
    where
    \[
      c_5 := \log^+(Kc_{3})+\log\theta_{1}\ \textrm{ and }\ c_{6} := c_{5}/\log 2.
    \]
    Let us recall $N_0$ as in Lemma~\ref{Binet-type}. If $n_1<N_0$, then $y^a \ll_{\alpha} K$ so that Theorems~\ref{Th:Ham} and \ref{Th:y} follow trivially. We may, therefore, assume $n_1 \geq N_0$ so that the lower bound in~\eqref{qn_ub_lb} applies and
    \[
      y^a\ \geq\ q^{(j_1)}_{n_1}\ \geq\ c_{4}\theta_{1}^{n_1}.
    \]
    Now, let~\eqref{E:ya} and~\eqref{E:y} hold together. Denote $\varphi$ to be the golden rat{i}o $(1+\sqrt{5}) / 2$. It is well-known and can in part{i}cular be proved using induction that
    \[
      F_t = \frac{\varphi^t-(-\varphi)^{-t}}{\sqrt{5}} 
    \]
    which implies $\varphi^{t-2} \leq F_t \leq \varphi^{t-1}$ for all $t \geq 1$. From \eqref{E:y} and the greedy property of Zeckendorf representat{i}ons~\eqref{greediness}, one has $y < F_{m_1+1} \leq \varphi^{m_1}$ giving us
    \begin{equation}\label{E:logy_ub_m1}
      \log y\ <\ m_1 \log\varphi\ <\ m_1.
    \end{equation}
    We also have
    \[
      \varphi^{m_1-2}\ \leq\ F_{m_1}\ \leq\ y\ <\ y^a\ \leq\ Kc_{3}\theta_{1}^{n_1}
    \]
    so that
    \begin{equation}\label{E:m1_ub_n1}
      m_1\ \leq\ c_{7} n_1,\quad\textrm{where }\ c_{7} := \frac{\log^+(K\varphi^2 c_{3})+\log\theta_{1}}{\log\varphi}.
    \end{equation}
    Further,
    \[
      c_{4}\theta_{1}^{n_1}\ \leq\ y^a\ <\ F_{m_1+1}^a\ \leq\ \varphi^{am_1}
    \]
    implying
    \begin{equation}\label{E:n1_ub_m1}
      n_1\ \leq\ c_{8} \cdot am_1\quad\textrm{where }\ c_{8} := \max\,\big\{\,\frac{\log\varphi+\log^+(1/c_{4})}{\log\theta_{1}},\,1\,\big\}.
    \end{equation}
    Lastly, we may have occasion to employ the following result due to Peth\"{o} and de Weger~\cite{Petho_Weger_86} while `transferring secondary factors.'
    \begin{lemma}\label{L:PW}
      Let $a \geq 0,\ c \geq 1,\ g > \big( e^2 / c \big)^{c}$ and $x \in \mathbb{R}$ be the largest solution of $x= a+g(\log x)^{c}.$ Then,
      \[
        x\ <\ 2^{c}\,\big(\,a^{1/c} + g^{1/c} \log (c^{c}g)\,\big)^{c}.
      \]
    \end{lemma}
    
  \subsection{The case $k = 1$}\label{SubS:k1}
  If so, Equat{i}on~\eqref{E:yaf9} transforms as
  \begin{equation}\label{E:keq1}
    Kq_{N_1} = y^a.
  \end{equation}
  Consider prime decomposi{t}ions
  \[
    K = p_1^{e_1} \cdots p_t^{e_t} \text{ and } y = p_1^{f_1} \dots p_t^{f_t}y_1 \text{ where } \gcd\,(K,y_1) = 1.
  \]
  For the triple $( y, a, N_1 )$ to be a solut{i}on of~\eqref{E:keq1}, each $f_j \geq 1$ and $af_j \geq e_j$ for $1 \leq j \leq t$. At this stage, one has
  \[
    q_{N_1} \in \mathcal{S} y_1^a
  \]
  where $\mathcal{S}$ is the set of all non-zero integers with same prime factors as $K$.
  Our claim is that Theorems~\ref{Th:Ham} and \ref{Th:y} now follow from~\cite[\S\,1]{Pe_82}:
  \begin{proposition}[Peth\"{o}]\label{P:Pet}
    Let $A, B, G_0$ and $G_1$ be integers such that $A \big(\,|G_0| + |G_1|\,\big) \neq 0,\ \gcd ( A, B ) = 1$ and $A^2 / B \notin \{ 1, 2, 3, 4 \}$. Furthermore, assume that $A^2 - 4B$ is not a perfect square whenever $B ( G_1^2 - AG_0G_1 + BG_0^2 ) = 0$. Def{i}ne linear recursive sequence $( G_n )_{n\,\geq\,0}$ as
    \[
      G_{n + 2} = AG_{n + 1} - B G_{n}\quad\text{for all } n \geq 0
    \]
    and $\mathcal{S}$ be the set of all non-zero integers with prime factors belonging to some f{i}xed f{i}nite set. Then, the Diophant{i}ne equat{i}on $G_n = wx^a$ with $a \geq 2$  and $w \in \mathcal{S}$ has f{i}nitely many integer solut{i}ons bounded as per
    \begin{align*}
      \max \big\{ |w|, |x|, n, a \big\}\ \ll_{A, B, G_0, G_1, \mathcal{S}}\ 1 &\quad\text{if } x > 1, \text{ and}\\
      \max \big\{ |w|, n \big\}\ \ll_{A, B, G_0, G_1, \mathcal{S}}\ 1 &\quad\text{if } x = 1.
    \end{align*}
    Here, the implied constants are ef{f}ect{i}vely computable.
  \end{proposition}
  \noindent The requisite condi{t}ions in Proposi{t}ion~\ref{P:Pet} are always sat{i}sf{i}ed for us as $t_{\alpha}^2-4(-1)^s$ is never a perfect square in~\eqref{E:rec}. If $t_{\alpha} \in \{ 1, 2 \}$, the period $s$ is forced to be equal to $1$ so that $t_{\alpha}^2 - 4 ( -1 )$ is not a square. Else, the trace of the product integer matrix in~\eqref{t_alpha} is at least $3$ and any other perfect square has distance at least $5$ from $t_{\alpha}^2$. Moreover,
  \[
    t_{\alpha}^2\neq j(-1)^s \text{ for } 1 \leq j \leq 4.
  \]
  The bot{t}omline is that $\max \{ y,a,N_1 \} \ll_{\alpha} 1$. Thus, we may hereaf{t}er assume $k\geq 2$.

  \section{Bound in terms of $y$}\label{sec_thm2}
    \noindent Let $w\in\{1,\ldots,k\}$, where $k \geq 2$ and set $n_{k + 1} = 0$. Equat{i}ons~\eqref{E:binet_rep} and~\eqref{E:ya_modified} can be combined to write
    \[
      y^a - \sum_{i\,=\,1}^w d_ic_{1}^{(j_i)} \theta_{1}^{n_i}\ =\ -\sum_{i\,=\,1}^w d_ic_{2}^{(j_i)} \theta_{2}^{n_i} + \sum_{i\,=\,w\,+\,1}^{k} d_iq^{(j_i)}_{n_i}.
    \]
    This along with the growth bounds given in~\eqref{qn_ub_lb}, the fact that $|\theta_{2}| < 1 < \theta_1$ and $c_{1}^{(j)}$'s being posi{t}ive for all $0 \leq j < s$ implies that
    \begin{align}
      \left| \dfrac{y^a}{\theta_{1}^{n_1} \sum_{i\,=\,1}^w d_ic_{1}^{(j_i)}\theta_{1}^{n_i-n_1}} - 1 \right|\ &\leq\ \frac{\sum_{i\,=\,1}^w d_i \big| c_{2}^{(j_i)}\theta_{2}^{n_i} \big|}{c_1^{(j_1)}\theta_{1}^{n_1}}\, +\, \frac{c_{3}}{c_1^{(j_1)}}\sum_{i\,=\,w\,+\,1}^{k}\frac{d_i}{\theta_1^{n_1-n_i}}\notag\\
      &\ll_{\alpha,k} \frac{1}{\theta_1^{n_1-n_{w+1}}}.\label{E:Thm2_lfl_ub}
    \end{align}
    We def{i}ne
    \begin{equation}\label{E:defn_gamma_Aw}
      \Gamma_{A,\,w}\ :=\ y^a \theta_{1}^{-n_1} \left(\,\sum_{i\,=\,1}^w d_ic_{1}^{(j_i)}\theta_{1}^{n_i-n_1}\,\right)^{-1}.
    \end{equation}
    \noindent Suppose $\Gamma_{A,\,w}=1$. It is possible only if
    \[
      y^a\ =\ \sum_{i\,=\,1}^w d_ic_{1}^{(j_i)}\theta_{1}^{n_i}.
    \]
    Consider the conjugate embedding $\sigma : \mathbb{Q}(\theta_{1}) \to \mathbb{C}$ mapping $\theta_{1}$ to $\theta_{2}$. Af{t}er verifying that $\sigma \big( c_1^{(j)} \big) = -c_2^{(j)}$ for all $0 \leq j < s$, one gets
    \[
      \sum_{i\,=\,1}^w d_ic_{1}^{(j_i)}\theta_{1}^{n_i}\ =\ y^a\ =\ \sigma ( y^a )\ =\ \sigma \big(\,\sum_{i\,=\,1}^w d_ic_{1}^{(j_i)}\theta_{1}^{n_i}\,\big)\ =\ -\sum_{i\,=\,1}^w d_ic_{2}^{(j_i)}\theta_{2}^{n_i}.
    \]
    As $c_1^{(j)} > 0$ for all $j$ and $d_i \geq 1$ for all $1 \leq i \leq k$, we obtain
    \begin{equation}\label{E:n1ll1}
      c_{1}^{(j_1)}\theta_{1}^{n_1}\ <\ \left|\,\sum_{i\,=\,1}^w d_ic_{2}^{(j_i)}\theta_{2}^{n_i}\,\right|\ \ll_{\alpha,\,K} 1
    \end{equation}
    whereby $n_1\ll_{\alpha,\,K} 1$. We may then invoke~\eqref{E:a_n1} to argue that the size of any solut{i}on to~\eqref{E:ya} is bounded above by some ef{f}ect{i}vely computable constant depending only on $\alpha$ and $K$. Consequently, Theorem~\ref{Th:y} holds in this case.\\[-0.1cm]
    
    \noindent Now, suppose $\Gamma_{A,w} \neq 1$. If the upper bound achieved in~\eqref{E:Thm2_lfl_ub} happens to be greater than $1/2$, then $n_1-n_{w+1} \ll_{\alpha,\,K} 1$. Else, we apply Proposition~\ref{lfl} with 
    $T=3$, $D=2$,
    \begin{align*}
      \delta_1\ =\ y,\quad&\delta_2\ =\ \theta_{1},\quad\delta_3\ =\ \sum_{i\,=\,1}^w d_ic_{1}^{(j_i)}\theta_{1}^{n_i-n_1},\\
      k_1\ =\ a,\quad&k_2\ =\ -n_1\quad\text{and } k_3\ =\ -1.
    \end{align*}
    \noindent Since absolute logarithmic height $h(y) = \log y$ with $y \geq 2$ and $h ( \theta_{1} ) = ( \log\theta_{1} ) / 2$, we can take $A_1=2\log y$ and
    \[
      A_2\ :=\ \max\,\{\,\log\theta_{1},\,0.16\,\}\ =\ \log \theta_1
    \]
    recalling the discussion below Proposi{t}ion~\ref{P:Pet}. For the last algebraic number $\delta_3$, our recourse is Proposition~\ref{height_of_poly} read along with~\eqref{height_of_product}. Let
    \[
      f ( X_1, \ldots, X_w, Y ) = d_1 X_1 + d_2 X_2Y^{n_1 - n_2} + \cdots + d_w X_w Y^{n_1 - n_w}.
    \]
    One can easily see that $\delta_3 = f\,\big(\,c_1^{(j_1)},\,\ldots,\,c_1^{(j_w)},\,\theta_1^{-1}\,\big)$ and
    \begin{equation}
      h(\delta_3)\ \leq\ \sum_{i\,=\,1}^w h \big( c_{1}^{(j_i)} \big) + ( n_1 - n_{w} )\,h ( \theta_{1} ) +\log \big( \sum_{i\,=\,1}^{w} d_i\,\big).\label{E:height_of_delta3}
    \end{equation}
    Denote the bound obtained on the right side of~\eqref{E:height_of_delta3} by $\widetilde{A}_3$ and
    \[
      A_3 := \max\,\big\{\,2\widetilde{A}_3,\,|\log\delta_3|,\,0.16\,\big\}.
    \]
    Note that $d_1c_1^{(j_1)} < \delta_3 \ll_{\alpha} k$ which implies $| \log\delta_3 | \ll_{\alpha, K} 1$. In the light of~\eqref{E:a_n1}, we also def{i}ne $B := \max\,\{\,c_{6},\,1\,\}\cdot n_1 \geq \max\,\{\,a,\,n_1,\,1\,\}$. Thereaf{t}er, Matveev tells us
    \[
      \log|\Gamma_{A,w}-1|\ \geq \ -c_9\cdot\max\{n_1-n_w,1\}\log y \log(eB).
    \]
    On combining this with~\eqref{E:Thm2_lfl_ub}, we get
    \begin{equation}\label{E:UB_n1nw}
      n_1 - n_{\,w\,+\,1}\ \leq\ c_{10}\cdot\max\,\{\,n_1 - n_w,\,1\,\} \log y \log n_1
    \end{equation}
    where the constant $c_{10}$ is taken to be greater than $e^2 / \log 2$. This caut{i}on is being exercised awai{t}ing an imminent application of Lemma~\ref{L:PW}. It is worth point{i}ng out that the above inequality holds even when the upper bound in~\eqref{E:Thm2_lfl_ub} is more than $1/2$.\\[-0.1cm]
    
    \noindent We have already seen that Theorem~\ref{Th:y} follows from~\eqref{E:n1ll1} as soon as $\Gamma_{A,w} = 1$ for some $w\in\{1,\ldots,k\}$. Hence, assume $\Gamma_{A,w}\neq1$ for all $1 \leq w \leq k$. Applying~\eqref{E:UB_n1nw} repeatedly for $w = k, \ldots, 1$, one gets
    \begin{equation}\label{E:nlogn}
      n_1\ =\ n_1 - n_{k+1}\ \leq\ c_{10}^k ( \log y \log n_1 )^k.
    \end{equation}
    With $a = 0,\ c = k$ and $g = ( c_{10} \log y )^k$, the secondary factors on the right side of~\eqref{E:nlogn} can be transferred using Lemma~\ref{L:PW} to have
    \begin{equation}\label{E:thm2_conclusion}
      n_1\ \ll_{\alpha,k} ( \log^+\!y \cdot \log^+ \log^+\!y )^{k}.
    \end{equation}
Since $k\leq K$, Theorem~\ref{Th:y} follows. \hfill \qed

\section{Bound in terms of Zeckendorf Hamming weight of $y$}\label{sec_thm1}
  \noindent In view of \S \ref{SubS:k1}, we assume $k\geq 2$. It follows from \eqref{E:logy_ub_m1} and \eqref{E:thm2_conclusion} that 
  \begin{equation}\label{E:surp}
    n_1\ \ll_{\alpha, k}\ m_1^{k} ( \log^+\!m_1 )^k.
  \end{equation}
  A total of $k$ \textit{linear forms in logarithms}, namely $\Gamma_{A,w}$ for $1 \leq w \leq k$, were earlier constructed in~\S\,\ref{sec_thm2} using~\eqref{E:ya}. Similarly, $\ell$ linear forms in logarithms can be conjured up from~\eqref{E:y}. This has been a central idea in~\cite[\S\,3.1]{VZ_24}. As such, we only write the inequality of relevance to us. For $v \in \{ 1,\,\ldots,\,\ell\,\}$, let
  \[
    \Gamma_{B,\,v} := y \sqrt{5} \varphi^{-m_1} \left(\sum_{i=1}^v\varphi^{m_i-m_1}\right)^{-1}
  \]
  akin to~\eqref{E:defn_gamma_Aw}. Equat{i}on~\eqref{E:y} in Theorem~\ref{Th:Ham} may be used to have
    \begin{equation}\label{E:ub_gamma_bv}
        |\Gamma_{B,v}-1|\ \leq\ \frac{6}{\varphi^{m_1-m_{v+1}}},
    \end{equation}
    where $m_{\ell+1}$ is set to be $0$. For $w \in \{ 1,\,\ldots,\,k \}$ and $v \in \{ 1,\,\ldots,\,\ell\,\}$, let
    \begin{align}
      \Lambda_{A,\,w} &:= \log\Gamma_{A,\,w}\ =\ a\log y-n_1\log\theta_{1}-\log(\sum_{i=1}^w d_ic_{1}^{(j_i)}\theta_{1}^{n_i-n_1}),\label{E:AwBv}\\
      \Lambda_{B,\,v} &:= \log\Gamma_{B,\,v}\ =\ \log y+\log\sqrt{5}-m_1\log\varphi-\log(\sum_{i=1}^v\varphi^{m_i-m_1}), \text{ and}\notag\\
      \Lambda^*_{v,\,w} &:= a\Lambda_{B,\,v}-\Lambda_{A,\,w}\ =\ a\log\sqrt{5}-am_1\log\varphi-a\log(\sum_{i=1}^v\varphi^{m_i-m_1})\notag\\
      &\qquad\qquad\qquad\qquad\qquad + n_1\log\theta_{1}+\log(\sum_{i=1}^w d_ic_{1}^{(j_i)}\theta_{1}^{n_i-n_1}).\notag
    \end{align}
    Recall we mentioned in~\S\,\ref{sec_thm2} that if the right side of~\eqref{E:Thm2_lfl_ub} is more than $1/2$, then $n_1-n_{w\,+\,1}\ll_{\alpha,\,k} 1$. If not, the fact that
    \begin{equation}\label{E:G2L}
      | \log x\,|\ \leq\ 2\,|\,x - 1\,|\quad\text{whenever}\quad|\,x - 1\,|\ \leq\ 1/2
    \end{equation}
    can be employed to get
    \begin{equation}\label{E:Aw_ub}
      |\,\Lambda_{A,\,w}\,|\ \ll_{\alpha,\,k}\ \frac{1}{\theta_1^{n_1 - n_{w+1}}}.
    \end{equation}
    If $m_1-m_{v+1}\geq 6$, we similarly use \eqref{E:ub_gamma_bv} to obtain
    \[
      |\,\Lambda_{B,\,v}\,|\ \leq\ \frac{12}{\varphi^{m_1-m_{v+1}}}.
    \]
    Together, the last two inequalities above yield     
    \begin{equation}\label{E:ub_lambda_star}
      |\,\Lambda^*_{v,w}\,|\ \ll_{\alpha,\,k}\ \frac{a}{(\min\{\theta_1,\varphi\})^{\min\{ m_1 - m_{v + 1}, n_1 - n_{w + 1} \}}}
    \end{equation}
    under the assumpt{i}on that $m_1 - m_{v+1} \geq 6$ and $n_1-n_{w\,+\,1} \gg_{\alpha,\,k} 1$.
    
    \subsection{Non-vanishing of linear forms}\label{SS:nonzero}
      Suppose $\Lambda^*_{v,w}=0$. Then,
    \begin{equation}\label{E:lambda*}
      \left(\frac{1}{\sqrt{5}}\sum_{i=1}^v\varphi^{m_i}\right)^a=\sum_{i=1}^w d_ic_{1}^{(j_i)}\theta_{1}^{n_i}.
    \end{equation}
    As the lef{t} side of~\eqref{E:lambda*} lies in $\mathbb{Q}(\sqrt{5})$, the right side lies in $\mathbb{Q}(\theta_{1})$ and our assumpt{i}on was that $\mathbb{Q}(\sqrt{5}) \neq \mathbb{Q}(\theta_{1}) = \mathbb{Q} (\alpha)$, it follows that we have a rational number on both sides. Therefore,
    \[
      \big(\,\sum_{i=1}^v\varphi^{m_i}\,\big)^a\ =\ r_0(\sqrt{5} )^x\quad\text{for some } r_0\in\mathbb{Q} \text{ and } x \in \{ 0, 1 \}.
    \]
    Here, the lef{t} side is a polynomial in $\varphi$ with posi{t}ive integer coefficients.  We can also invoke the subst{i}tut{i}on $\varphi^2=\varphi+1$ to have
    \[
      e_0\frac{(1+\sqrt{5})}{2}+e_1\ =\ r_0(\sqrt{5})^x\quad\text{for some } e_0, e_1 \in \mathbb{N}.
    \]
    This is, however, impossible. In other words, $\Lambda^*_{v,w}\neq 0$ for all $v, w$. \hfill \qed\\
    
    \noindent We are now in a posi{t}ion to apply Proposition~\ref{lfl} with 
    $T=5,\ D = 4$,
    \begin{align*}
        &\delta_1 = \sqrt{5},\ \delta_2 = \varphi,\ \delta_3 = \sum_{i=1}^w d_ic_{1}^{(j_i)}\theta_{1}^{n_i-n_1},\ \delta_4 = \theta_{1},\ \delta_5 = \sum_{i=1}^v\varphi^{m_i-m_1},\\
        &k_1\ =\ a,\ k_2\ = -am_1,\ k_3\ = 1,\ k_4\ =\ n_1 \text{ and } k_5\ =\ -a.
    \end{align*}
    Let us ponder over the heights of some algebraic numbers given above. While $h ( \delta_1 )$ and $h ( \delta_2 )$ are absolute constants, we obtain from~\eqref{E:height_of_delta3} that
    \begin{equation}\label{E:hd3}
      h ( \delta_3 )\ \ll_{\alpha,\,K}\ w\cdot\max\,\{\,n_1 - n_w,\,1\,\}.
    \end{equation}
    Note that $h  ( \delta_4 ) = ( \log \theta_1 ) / 2$ and Proposition~\ref{height_of_poly} read along with~\eqref{height_of_product} tells us
    \[
      h(\delta_5)\ \leq\ (m_1-m_v)h(\varphi)+\log v\ \leq\ 2v(m_1-m_v).
    \]
    This bound is obeyed trivially when $v = 1$. In the light of \eqref{E:n1_ub_m1}, we may take $B = c_{8} am_1$ so as to have $k_j \leq B$ for all $1 \leq j \leq 5$. Thereaf{t}er,~\eqref{E:a_n1} and \eqref{E:m1_ub_n1} will help us to deduce that
    \[
      \log|\Lambda^*_{v,w}|\ \geq -c_{11} \cdot vw \cdot \max\{m_1-m_v,1\} \cdot \max\{n_1-n_w,1\}\log n_1.
    \]
    On comparing this lower bound with \eqref{E:ub_lambda_star}, one gets
    \begin{align}
     \min\,\{\,&m_1 - m_{v + 1},\,n_1 - n_{w + 1}\,\}\label{E:ub_wv_plus1}\\
             &\ll_{\alpha,k} vw\max\{m_1-m_v,1\}\max\{n_1-n_w,1\}\log n_1.\notag
    \end{align}
    To summarize, we have $m_1-m_{v+1}\leq 6$ or $n_1-n_{w+1}\ll_{\alpha,k} 1$ or \eqref{E:ub_wv_plus1} holds true. It must be noted that the f{i}rst two of these possibilit{i}es may be subsumed by the third one af{t}er having a large enough value of the implied mult{i}plicat{i}ve constant.\\[-0.1cm]

    \noindent Ini{t}iate $u(0) = 1$. Let us introduce a double-indexed step counter $( v_j, w_j ),\  j \geq 1$. For $j=1$, we define $( v_1, w_1 )=( 2, 2)$. This corresponds to node $( 1, 1)$ in F{i}gure~\ref{F:flo}.  As hinted before, we begin with an auxiliary upper bound
    \[
      \min\,\{\,m_1 - m_2,\,n_1 - n_2\,\}\ \leq\ C_{12}\log n_1 =: u(1).
    \]
    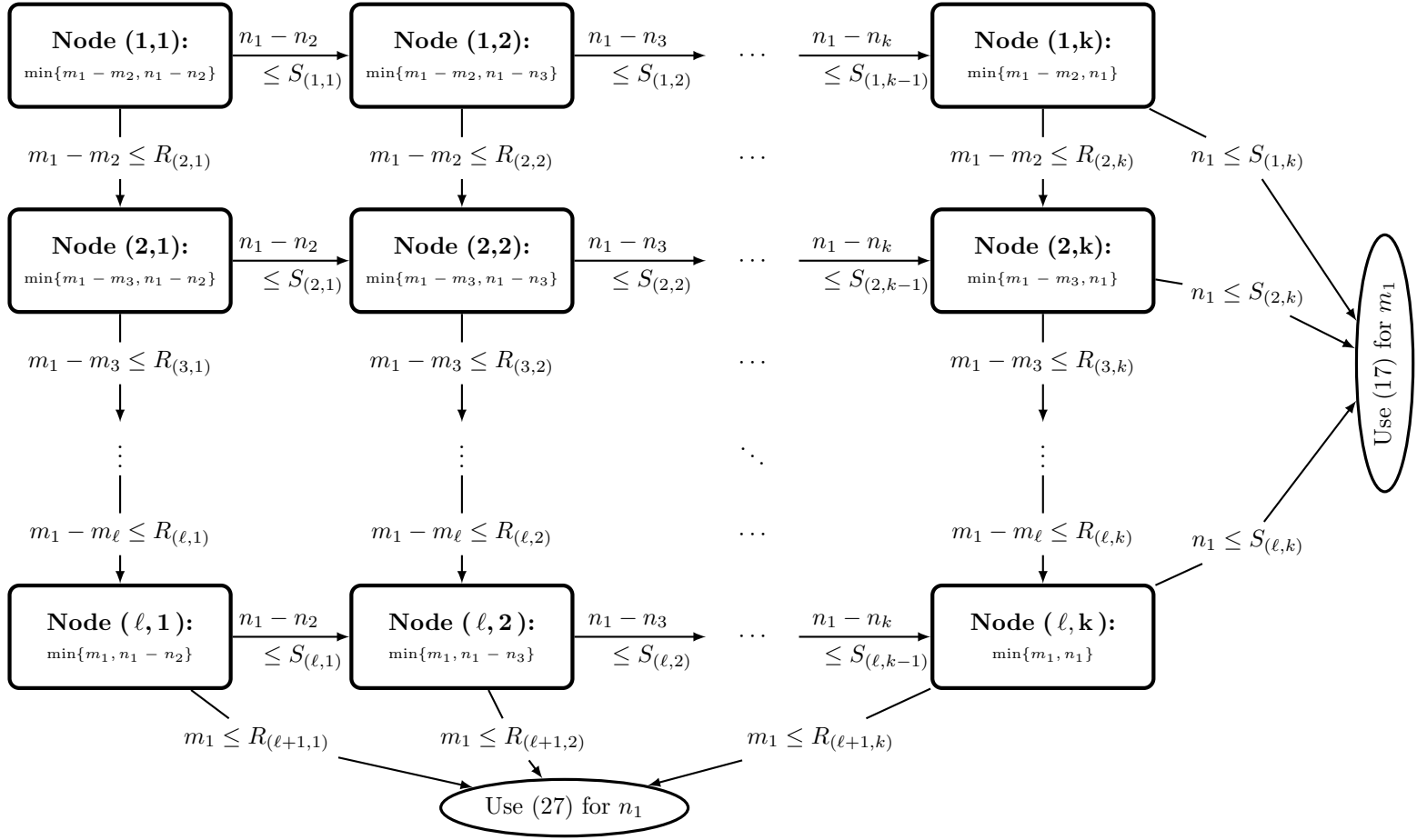
\begin{figure}[p]
      \begin{tikzpicture}[rotate=90,transform shape]
        \node (11) [steps] {{\bf Node (1,1):}\\ \tiny{$\min \{ m_1 - m_2, n_1 - n_2 \}$}};
        \node (12) [steps, right of=11,xshift=4cm] {{\bf Node (1,2):}\\ \tiny{$\min \{ m_1 - m_2, n_1 - n_3 \}$}};
        \draw [thick,-latex] (11) -- node [anchor=south] {$n_1 - n_2\quad$} node[anchor=north] {$\quad\leq S_{( 1, 1 )}$} (12);
        \node (1121) [mid,below of=11,yshift=-0.5cm] {$m_1 - m_2 \leq R_{( 2, 1 )}$};
        \draw [thick] (11) -- (1121);
        \node (21) [steps, below of=1121,yshift=-0.5cm] {{\bf Node (2,1):}\\ \tiny{$\min \{ m_1 - m_3, n_1 - n_2 \}$}};
        \draw[thick,-latex] (1121) -- (21);
        \node (21v1) [mid,below of=21,yshift=-0.5cm] {$m_1 - m_3 \leq R_{( 3, 1 )}$};
        \draw [thick] (21) -- (21v1);
        \node (v1) [mid,below of=21v1,yshift=-0.25cm] {$\vdots$};
        \draw[thick,-latex] (21v1) -- (v1);
        \node (v1l1) [mid,below of=v1,yshift=-0.25cm] {$m_1 - m_{\ell} \leq R_{( \ell, 1 )}$};
        \draw [thick] (v1) -- (v1l1);
        \node (l1) [steps,below of=v1l1,yshift=-0.5cm] {{\bf Node (\,$\mathbf{\ell}$,\,1\,):}\\ \tiny{$\min \{ m_1, n_1 - n_2 \}$}};
        \draw[thick,-latex] (v1l1) -- (l1);

        \node (1222) [mid,right of=1121,xshift=4cm] {$m_1 - m_2 \leq R_{( 2, 2 )}$};
        \draw [thick] (12) -- (1222);
        \node (22) [steps,right of=21,xshift=4cm] {{\bf Node (2,2):}\\ \tiny{$\min \{ m_1 - m_3, n_1 - n_3 \}$}};
        \draw [thick,-latex] (21) -- node [anchor=south] {$n_1 - n_2\quad$} node [anchor=north] {$\quad\leq S_{( 2, 1 )}$} (22);
        \draw[thick,-latex] (1222) -- (22);
        \node (22v2) [mid,below of=22,yshift=-0.5cm] {$m_1 - m_3 \leq R_{( 3, 2 )}$};
        \draw [thick] (22) -- (22v2);
        \node (v2) [mid,below of=22v2,yshift=-0.25cm] {$\vdots$};
        \draw[thick,-latex] (22v2) -- (v2);
        \node (v2l2) [mid,below of=v2,yshift=-0.25cm] {$m_1 - m_{\ell} \leq R_{( \ell, 2 )}$};
        \draw [thick] (v2) -- (v2l2);
        \node (l2) [steps,below of=v2l2,yshift=-0.5cm] {{\bf Node (\,$\mathbf{\ell}$,\,2\,):}\\ \tiny{$\min \{ m_1, n_1 - n_3 \}$}};
        \draw[thick,-latex] (v2l2) -- (l2);
        \draw[thick,-latex] (l1) -- node [anchor=south] {$n_1 - n_2\quad$} node [anchor=north] {$\quad\leq S_{( \ell, 1 )}$} (l2);

        \node (1w) [mid,right of=12,xshift=3.25cm] {$\quad\cdots\quad$};
        \draw [thick,-latex] (12) -- node [anchor=south] {$n_1 - n_3\quad$} node [anchor=north] {$\quad\leq S_{( 1, 2 )}$} (1w);

        \node (2w) [mid,right of=22,xshift=3.25cm] {$\quad\cdots\quad$};
        \draw [thick,-latex] (22) -- node [anchor=south] {$n_1 - n_3\quad$} node [anchor=north] {$\quad\leq S_{( 2, 2 )}$} (2w);

        \node (1k) [steps,right of=1w,xshift=3.25cm] {{\bf Node (1,$\mathbf{k}$):}\\ \tiny{$\min \{ m_1 - m_2, n_1 \}$}};
        \draw [thick,-latex] (1w) -- node [anchor=south] {$n_1 - n_k\quad$} node [anchor=north] {$\quad\leq S_{( 1, k - 1 )}$} (1k);

        \node (2k) [steps,right of=2w,xshift=3.25cm] {{\bf Node (2,$\mathbf{k}$):}\\   \tiny{$\min \{ m_1 - m_3, n_1 \}$}};
        \draw [thick,-latex] (2w) -- node [anchor=south] {$n_1 - n_k\quad$} node [anchor=north] {$\quad\leq S_{( 2, k - 1 )}$} (2k);
    
        \node (1k2k) [mid,below of=1k,yshift=-0.5cm] {$m_1 - m_2 \leq R_{( 2, k )}$};
        \draw [thick] (1k) -- (1k2k);
        \draw[thick,-latex] (1k2k) -- (2k);
        \node [mid,right of=1222,xshift=3.25cm] {$\quad\cdots\quad$};
        \node (2kvk) [mid,below of=2k,yshift=-0.5cm] {$m_1 - m_3 \leq R_{( 3, k )}$};
        \draw [thick] (2k) -- (2kvk);
        \node (vk) [mid,below of=2kvk,yshift=-0.25cm] {$\vdots$};
        \draw[thick,-latex] (2kvk) -- (vk);

        \node (vklk) [mid,below of=vk,yshift=-0.25cm] {$m_1 - m_{\ell} \leq R_{( \ell, k )}$};
        \draw [thick] (vk) -- (vklk);
        \node (lk) [steps,below of=vklk,yshift=-0.5cm] {{\bf Node (\,$\mathbf{\ell, k}$\,):}\\ \tiny{$\min \{ m_1, n_1 \}$}};
        \draw[thick,-latex] (vklk) -- (lk);

        \node [mid,right of=22v2,xshift=3.25cm] {$\quad\cdots\quad$};

        \node (vw) [mid,right of=v2,xshift=3.25cm] {$\quad\ddots\quad$};
        \node (lw) [mid,right of=l2,xshift=3.25cm] {$\quad\cdots\quad$};
        \draw [thick,-latex] (l2) -- node [anchor=south] {$n_1 - n_3\quad$} node [anchor=north] {$\quad\leq S_{( \ell, 2 )}$} (lw);
        \draw[thick,-latex] (lw) -- node [anchor=south] {$n_1 - n_k\quad$} node [anchor=north] {$\quad\leq S_{( \ell, k - 1 )}$} (lk);

        \node [mid,right of=v2l2,xshift=3.25cm] {$\quad\cdots\quad$};
      
        \node (endbelow) [endbox,below of=l2,xshift=1.5cm,yshift=-1.5cm] {Use~\eqref{E:surp} for $n_1$};

        \node (lp11) [mid,below of=l1,yshift=-0.5cm,xshift=2cm] {$m_1 \leq R_{( \ell + 1, 1 )}$};
        \draw[thick] (l1) -- (lp11);

        \node (lp12) [mid,below of=l2,yshift=-0.5cm,xshift=0.75cm] {$m_1 \leq R_{( \ell + 1, 2 )}$};
        \draw[thick] (l2) -- (lp12);
      
        \node (lp1k) [mid,below of=lk,yshift=-0.5cm,xshift=-3.25cm] {$m_1 \leq R_{( \ell + 1, k )}$};
        \draw[thick] (lk) -- (lp1k);
      
        \draw[thick,-latex] (lp11) -- (endbelow);
        \draw[thick,-latex] (lp12) -- (endbelow);
        \draw[thick,-latex] (lp1k) -- (endbelow);

        \node (endright) [endbox,right of=2kvk,xshift=4cm,rotate=90]{Use~\eqref{E:m1_ub_n1} for $m_1$};

        \node (1kp1) [mid,right of=1k,yshift=-1.5cm,xshift=2cm] {$n_1 \leq S_{( 1, k )}$};
        \draw[thick] (1k) -- (1kp1);
        \draw[thick,-latex] (1kp1) -- (endright);

        \node (2kp1) [mid,right of=2k,yshift=-0.5cm,xshift=2cm] {$n_1 \leq S_{( 2, k )}$};
        \draw[thick] (2k) -- (2kp1);
        \draw[thick,-latex] (2kp1) -- (endright);

        \node (lkp1) [mid,right of=lk,yshift=1.375cm,xshift=2cm] {$n_1 \leq S_{( \ell, k )}$};
        \draw[thick] (lk) -- (lkp1);
        \draw[thick,-latex] (lkp1) -- (endright);
      \end{tikzpicture}
      \caption{Intermediate bounds to be computed}\label{F:flo}
    \end{figure}
    We now define the counter in an ad hoc fashion. If
    \[
      \min\,\{\,m_1 - m_{v_j},\,n_1 - n_{w_j}\,\}\ =\ m_1 - m_{v_j},
    \]
    we define $( v_{j+1}, w_{j+1} )=( v_j+1, w_j )$ and move to the node one step towards the bottom. Else, declare $( v_{j+1}, w_{j+1} )=( v_j, w_j+1 )$ and move to the node one step towards the right. Note that we have 
    \[
      v_{j + 1} + w_{j + 1} = v_j + w_j + 1,\ \text{ for all } j.
    \]
    Continue until 
    \[
    v_j=\ell+1\ \textrm{ and }\ m_1\leq n_1-n_{w_j}
    \]
    which causes us to exit via node $( \ell, w_j - 1 )$ belonging to the last row, or
    \[
    w_j=k+1\ \textrm{ and }\ n_1\leq m_1-m_{v_j}
    \]
    which leads to our exit via node $( v_j - 1, k )$ belonging to the last column of F{i}gure~\ref{F:flo}. It is clear that we will exit in at most $k+\ell - 1$ steps. For $j\geq 2$, define
    \begin{equation}\label{E:uj}
      u (j)\ =\ C_{12} ( v_j - 1 ) ( w_j - 1 ) u ( j - 1 ) u ( j - 2 ) \log n_1.
    \end{equation}
    On reaching node $( v_j - 1, w_j - 1 )$, one encounters $\min \{ m_1 - m_{v_j}, n_1 - n_{w_j} \}$ in~\eqref{E:ub_wv_plus1} where each of the two $\max$ terms on the right must have been necessarily computed at two earlier stages during the walk irrespect{i}ve of the path taken so far. The upper bound obtained at the $j$-th step of the walk does not exceed $u(j)$. We should be careful to f{i}x $C_{12}$ large enough so as to ensure $u(j + 1) \geq u(j) \geq 1$ for all $j$.
    \begin{lemma}\label{L:ub}
      For all $j \geq 0$, the upper bounds $u (j)$ sat{i}sfy
      \[
        u(j)\ \leq\ C_{12}^{F_{j + 2} - 1}  ( \ell k )^{F_{j + 1} - 1} ( \log n_1 )^{F_{j + 2} - 1}
      \]
      where $F_j$'s are F{i}bonacci numbers as indexed at the beginning of this art{i}cle.
    \end{lemma}
    \begin{proof}
      Our claim is clearly true for $j \leq 1$. Let us assume it to be true for all $0 \leq j \leq j_0$. Then, we have from~\eqref{E:uj} that
      \begin{align*}
        u ( j_0 + 1 )\ &=\ C_{12} ( v_{j_0 + 1} - 1 ) ( w_{j_0 + 1} - 1 ) u ( j_0 ) u ( j_0 - 1 ) \log n_1\\
        &\leq\ C_{12}\cdot\ell k\cdot u ( j_0 ) u ( j_0 - 1 ) \log n_1\\
        &\leq\ C_{12}^{1 + F_{j_0 + 2} - 1 + F_{j_0 + 1} - 1}  ( \ell k )^{1 + F_{j_0 + 1} - 1 + F_{j_0} - 1} ( \log n_1 )^{1 + F_{j_0 + 2} - 1 + F_{j_0 + 1} - 1}\\
        &\leq\ C_{12}^{F_{j_0 + 3} - 1}  ( \ell k )^{F_{j_0 + 2} - 1} ( \log n_1 )^{F_{j_0 + 3} - 1}.
      \end{align*}
      By the principle of mathemat{i}cal induct{i}on, we are done.
    \end{proof}

    \noindent Taking $j=k+\ell - 1$, we f{i}nd that $\min \{ m_1, n_1 \} \ll_{\alpha,k,\ell} (\log n_1)^{F_{k+\ell+1}}$. In case $m_1 \leq n_1$, we will use \eqref{E:surp} to have an upper bound for $n_1$ as well. Thereaf{t}er or otherwise,
    one can apply Lemma~\ref{L:PW} to get an effectively computable upper bound for $n_1$ depending only on $\alpha, K$ and $\ell$. It then becomes apparent from~\eqref{E:a_n1} that the size of any solut{i}on to the system of equations \eqref{E:ya} and \eqref{E:y} is bounded above by some ef{f}ect{i}vely computable constant which depends only on $\alpha, K$ and $\ell$. This concludes the proof of Theorem~\ref{Th:Ham}. \hfill \qed
    
  \section{Bound in terms of radix Hamming weight of $y$}\label{sec_thm3}
    \noindent Our proof of Theorem \ref{Th:Ham2} is analogous to that of Theorem \ref{Th:Ham}. As such, we only provide a brief outline here.\\[-0.1cm]

    \noindent Let us recall that~\eqref{E:ya} was rewritten as~\eqref{E:ya_modified} to have simultaneous equations~\eqref{E:ya_modified} and~\eqref{E:y2}. We first record some elementary inequalities in the same spirit as~\S\,\ref{SubS:EI}. The greedy property of $b$-ary expansions dictates that $y < b^{m_1 + 1}$ which begets
    \begin{equation}\label{E2:logy_ub_m1}
        \log y\ <\ ( m_1 + 1 ) \log b.
    \end{equation}
    Next, the chain of inequali{t}ies
     \[
      b^{m_1} \leq\  y\ <\ y^a\ \leq\ Kc_{3}\theta_{1}^{n_1}
    \]
    tells us
    \begin{equation}\label{E2:m1_ub_n1}
      m_1\ \leq\ c_{7}' n_1,\quad\textrm{where }\ c_{7}' := \frac{\log^+(K c_{3})+\log\theta_{1}}{\log b}.
    \end{equation}
    Similar bounds for the $a$-th power of $y$ are given by
   \[
      c_{4}\theta_{1}^{n_1}\ \leq\ y^a\ <\ b^{(m_1+1)a}
    \]
    from which we obtain
    \begin{equation}\label{E2:n1_ub_m1}
      n_1\ \leq\ c_{8}' \cdot a ( m_1 + 1 ),\quad\textrm{where }\ c_{8}' := \max\,\bigl\{\,\frac{\log b+\log^+(1/c_{4})}{\log\theta_{1}},\,1\,\bigr\}.
    \end{equation}
    In view of \S \ref{SubS:k1}, we assume $k\geq 2$. It now follows from~\eqref{E:thm2_conclusion} and~\eqref{E2:logy_ub_m1} that 
    \[
      n_1\ \ll_{\alpha,\,k,\,b}\ ( m_1 + 1 )^{k} \bigl(\,\log^+\! ( m_1 + 1 ) \bigr)^k.
    \]
    Following the same steps taken in \S \ref{sec_thm1}, we construct
 $\ell$ linear forms in logarithms  from~\eqref{E:y2}.  For $v \in \{ 1,\,\ldots,\,\ell\,\}$, let
  \[
    \Gamma_{B,\,v} := y D_1^{-1} b^{-m_1} \left(\sum_{i=1}^v D_iD_1^{-1}b^{m_i-m_1}\right)^{-1}\quad\text{and}\quad\Lambda_{B,\,v} := \log \Gamma_{B,\,v}.
  \]
 Thereaf{t}er, Equat{i}on~\eqref{E:y2} in Theorem~\ref{Th:Ham2} can be used to have
    \[
        \bigl|\,\Gamma_{B,v}-1\,\bigr|\ =\ \frac{\sum_{i=v+1}^{\ell} D_ib^{m_i}}{\sum_{i=1}^v D_ib^{m_i}}\ \leq\ \frac{b^{m_{v + 1} + 1}}{b^{m_1}}\ =\ \frac{b}{b^{m_1-m_{v+1}}},
    \]
    where $m_{\ell+1}$ is again set to be $0$. If $m_1-m_{v+1}\geq 2$, we leverage~\eqref{E:G2L} to get
    \begin{equation}\label{E:Bv_ub}
      |\,\Lambda_{B,\,v}\,|\ \leq\ \frac{2b}{b^{m_1-m_{v+1}}}.
    \end{equation}
For $w \in \{ 1,\,\ldots,\,k \}$ and $v \in \{ 1,\,\ldots,\,\ell\,\}$, let
    \begin{align*}
    \Lambda^*_{v,\,w} := a\Lambda_{B,\,v}-\Lambda_{A,\,w}\ &=\ -a\log D_1-am_1\log b-a\log(\sum_{i=1}^vD_iD_1^{-1}b^{m_i-m_1})\\
      &\qquad\qquad\qquad+\ n_1\log\theta_{1}+\log(\sum_{i=1}^w d_ic_{1}^{(j_i)}\theta_{1}^{n_i-n_1}),
      \end{align*}
      where $\Lambda_{A,\,w}$ remains as defined before in~\eqref{E:AwBv}.
   From~\eqref{E:Aw_ub} and \eqref{E:Bv_ub}, we obtain
    \begin{equation}\label{E2:ub_lambda_star}
      |\,\Lambda^*_{v,w}\,|\ \ll_{\alpha,\,K,\,b}\ \frac{a}{\bigl(\,\min\{\,b,\,\theta_1\,\}\,\bigr)^{\min\{m_1 - m_{v + 1},\,n_1 - n_{w + 1} \}}}
    \end{equation}
    under the assumpt{i}on that $m_1 - m_{v+1} \geq 2$ and $n_1-n_{w\,+\,1} \gg_{\alpha,\,K} 1$.

    \subsection{Non-vanishing of linear forms} Suppose $\Lambda^*_{v,w}=0$. It is possible only if
    \[
       \sum_{i\,=\,1}^w d_ic_{1}^{(j_i)}\theta_{1}^{n_i}\ =\ \bigl(\sum_{i=1}^v D_ib^{m_i} \bigr)^a\ \in\ \mathbb{Q}.
    \]
    Consider the conjugate embedding $\sigma : \mathbb{Q}(\theta_{1}) \to \mathbb{C}$ mapping $\theta_{1}$ to $\theta_{2}$. One gets
    \begin{align*}
      \sum_{i\,=\,1}^w d_ic_{1}^{(j_i)}\theta_{1}^{n_i}\ &=\ \bigl( \sum_{i=1}^v D_ib^{m_i} \bigr)^a\\
      &=\ \sigma (\sum_{i=1}^v D_ib^{m_i})^a\ =\ \sigma \big(\,\sum_{i\,=\,1}^w d_ic_{1}^{(j_i)}\theta_{1}^{n_i}\,\big)\ =\ -\sum_{i\,=\,1}^w d_ic_{2}^{(j_i)}\theta_{2}^{n_i}.
    \end{align*}
    As $c_1^{(j)} > 0$ for all $j$ and $d_i \geq 1$ for all $1 \leq i \leq k$, we obtain
    \[
      c_{1}^{(j_1)}\theta_{1}^{n_1}\ <\ \left|\,\sum_{i\,=\,1}^w d_ic_{2}^{(j_i)}\theta_{2}^{n_i}\,\right|\ \ll_{\alpha,\,K}\ 1
    \]
    whereby $n_1\ll_{\alpha,\,K} 1$. We may again utilize~\eqref{E:a_n1} to argue that the size of any solut{i}on to~\eqref{E:ya} is bounded above by some ef{f}ect{i}vely computable constant depending only on $\alpha$ and $K$. Consequently, Theorem~\ref{Th:y} holds in this case. Hence, we will hereaf{t}er assume that $\Lambda^*_{v,w}\neq 0$ for all $v, w$. \hfill \qed\\
    
    \noindent We can now apply Proposition~\ref{lfl} with 
    $T=5,\ D=2$,
    \begin{align*}
        &\delta_1 = D_1,\ \delta_2 = b,\ \delta_3 = \sum_{i=1}^w d_ic_{1}^{(j_i)}\theta_{1}^{n_i-n_1},\ \delta_4 = \theta_{1},\ \delta_5 = \sum_{i=1}^vD_iD_1^{-1}b^{m_i-m_1},\\
        &k_1\ =\ -a,\ k_2\ = -am_1,\ k_3\ = 1,\ k_4\ =\ n_1 \text{ and } k_5\ =\ -a.
    \end{align*}
    As in~\S\,\ref{sec_thm1}, the required heights $h ( \delta_1 )$ and $h ( \delta_2 )$ are absolute constants. Further, $h ( \delta_3 )$ has already been obtained  in~\eqref{E:hd3}, $h  ( \delta_4 ) = ( \log \theta_1 ) / 2$ and Proposition~\ref{height_of_poly} together with~\eqref{height_of_product} tells us that
    \[
      h ( \delta_5 )\ \leq\ ( m_1 - m_v ) \log b + 2\log\,\bigl(\,\sum_{i=1}^vD_i\,\bigr)\ \leq\ 5 ( m_1 - m_v ) \log^+\!b.
    \]
    In the light of \eqref{E2:n1_ub_m1}, we may take $B = c_{8}' a ( m_1 + 1 )$ so as to have $k_j \leq B$ for all $1 \leq j \leq 5$. Thereaf{t}er,~\eqref{E:a_n1} and \eqref{E2:m1_ub_n1} will help us to deduce that
    \[
      \log|\Lambda^*_{v,w}|\ \geq\ -c_{11} \cdot vw \cdot \max\{m_1-m_v,1\} \cdot \max\{n_1-n_w,1\}\log n_1.
    \]
    On comparing this lower bound against~\eqref{E2:ub_lambda_star}, one gets
    \begin{align}
     \min\,\{\,&m_1 - m_{v + 1},\,n_1 - n_{w + 1}\,\}\label{E2:ub_wv_plus1}\\
             &\ll_{\alpha,\,K,\,b} vw\max\{m_1-m_v,1\}\max\{n_1-n_w,1\}\log n_1.\notag
    \end{align}
    To summarize, we have $m_1-m_{v+1}\leq 2$ or $n_1-n_{w+1}\ll_{\alpha,k} 1$ or \eqref{E2:ub_wv_plus1} holds true. This is similar to our observat{i}on following~\eqref{E:ub_wv_plus1} and we can proceed exactly as in \S \ref{sec_thm1} to complete the proof of Theorem~\ref{Th:Ham2}. \hfill \qed

  \bibliographystyle{plain}
  \bibliography{SS26}

@article {AHR_21,
    AUTHOR = {Aboudja, Hyacinthe and Hernane, Mohand and Rihane, Salah
              Eddine and Togb\'e, Alain},
     TITLE = {On perfect powers that are sums of two {P}ell numbers},
   JOURNAL = {Period. Math. Hungar.},
  FJOURNAL = {Periodica Mathematica Hungarica. Journal of the J\'anos Bolyai
              Mathematical Society},
    VOLUME = {82},
      YEAR = {2021},
    NUMBER = {1},
     PAGES = {11--15},
      ISSN = {0031-5303,1588-2829},
   MRCLASS = {11D61 (11B39)},
  MRNUMBER = {4219400},
MRREVIEWER = {L\'aszl\'o\ Szalay},
       DOI = {10.1007/s10998-020-00342-1},
       URL = {https://doi.org/10.1007/s10998-020-00342-1},
}

@article{BRP_22,
AUTHOR={Bhoi, Pritam Kumar and Rout, Sudhansu Sekhar and Panda, Gopal
              Krishna},
TITLE={On the resolution of the Diophantine equation {$U_n+U_m=x^q$}},
JOURNAL={Ramanujan J.},
VOLUME = {66, article no.~36},
YEAR={2025}
}

@article{EEKT_24,
AUTHOR={Earp-Lynch, Benjamin and  Earp-Lynch, Simon and Kihel, Omar  and   Tiebekabe, P.},
TITLE={Powers as {F}ibonacci sums},
JOURNAL={Quaest.~Math.},
FJOURNAL={Quaestiones Mathematicae},
      VOLUME = {},
      YEAR = {2024},
    NUMBER = {},
     PAGES = {1--13},
}

@article{MS24,
      title={On the representat{i}on of an integer in {O}strowski and recurrence numerat{i}on systems}, 
      author={Mohit Mi{t}tal and Divyum Sharma},
      year={2025},
      volume={91},
      pages={546--567},
      url={https://doi.org/10.1007/s10998-025-00670-0},
      journal = {Period.\ Math.\ Hungar.}
}

@book {Petho_81,
    AUTHOR = {Peth\H{o}, A.},
     TITLE = {Perfect powers in second order recurrences},
 BOOKTITLE = {Topics in classical number theory, {V}ol. {I}, {II}
              ({B}udapest, 1981)},
    SERIES = {Colloq. Math. Soc. J\'anos Bolyai},
    VOLUME = {34},
     PAGES = {1217--1227},
 PUBLISHER = {North-Holland, Amsterdam},
      YEAR = {1984},
      ISBN = {0-444-86509-8},
   MRCLASS = {11B37},
  MRNUMBER = {781182},
MRREVIEWER = {C.\ R.\ Pranesachar},
}

@article {KL_21,
    AUTHOR = {Kihel, Omar and Larone, Jesse},
     TITLE = {On the nonnegative integer solutions of the equation {$F_n\pm
              F_m=y^a$}},
   JOURNAL = {Quaest. Math.},
  FJOURNAL = {Quaestiones Mathematicae. Journal of the South African
              Mathematical Society},
    VOLUME = {44},
      YEAR = {2021},
    NUMBER = {8},
     PAGES = {1133--1139},
      ISSN = {1607-3606,1727-933X},
   MRCLASS = {11B39 (11D61 11J86)},
  MRNUMBER = {4316638},
MRREVIEWER = {Mahadi\ Ddamulira},
       DOI = {10.2989/16073606.2020.1775155},
       URL = {https://doi.org/10.2989/16073606.2020.1775155},
}

@article {Petho_Weger_86,
    AUTHOR = {Peth\H{o}, A. and de Weger, B.~M.~M.},
     TITLE = {Products of prime powers in binary recurrence sequences. {I}. {T}he hyperbolic case, with an applicat{i}on to the generalized {R}amanujan-{N}agell equation.},
   JOURNAL = {Mathematics of Computation},
  FJOURNAL = {Journal of Number Theory},
    VOLUME = {47},
      YEAR = {1986},
    NUMBER = {176},
     PAGES = {713--727},
       DOI = {10.1090/S0025-5718-1986-0856715-5},
       URL = {https://doi.org/10.1090/S0025-5718-1986-0856715-5},
}

@article {CHHM+,
    AUTHOR = {Cordwell, Katherine and Hlavacek, Max and Huynh, Chi and
              Miller, Steven J. and Peterson, Carsten and Vu, Yen Nhi
              Truong},
     TITLE = {Summand minimality and asymptotic convergence of generalized {Z}eckendorf decompositions},
   JOURNAL = {Res.~Number Theory},
  FJOURNAL = {Research in Number Theory},
    VOLUME = {4},
      YEAR = {2018},
    NUMBER = {4},
     PAGES = {27~pp.}
}

@article {Pe_82,
    AUTHOR = {Peth\H{o}, Attila},
     TITLE = {Perfect powers in second order linear recurrences},
   JOURNAL = {J. Number Theory},
  FJOURNAL = {Journal of Number Theory},
    VOLUME = {15},
      YEAR = {1982},
    NUMBER = {1},
     PAGES = {5--13},
      ISSN = {0022-314X,1096-1658},
   MRCLASS = {10B25 (10A35)},
  MRNUMBER = {666345},
MRREVIEWER = {H.\ L.\ Abbott},
       DOI = {10.1016/0022-314X(82)90079-8},
       URL = {https://doi.org/10.1016/0022-314X(82)90079-8},
}

@article {Ze_72,
    AUTHOR = {Zeckendorf, E.},
     TITLE = {Repr\'esentation des nombres naturels par une somme de nombres de Fibonacci ou de nombres de Lucas},
   JOURNAL = {Bull. Soc. Roy. Sci. Li\`ege},
    VOLUME = {41},
      YEAR = {1972},
     PAGES = {179--182}
}

@article {BL_16,
    AUTHOR = {Bravo, Jhon J. and Luca, Florian},
     TITLE = {On the {D}iophantine equation {$F_n+F_m=2^a$}},
   JOURNAL = {Quaest. Math.},
  FJOURNAL = {Quaestiones Mathematicae. Journal of the South African
              Mathematical Society},
    VOLUME = {39},
      YEAR = {2016},
    NUMBER = {3},
     PAGES = {391--400},
      ISSN = {1607-3606,1727-933X},
   MRCLASS = {11B39 (11D61 11J86)},
  MRNUMBER = {3510734},
MRREVIEWER = {L\'aszl\'o\ Szalay},
       DOI = {10.2989/16073606.2015.1070377},
       URL = {https://doi.org/10.2989/16073606.2015.1070377},
}

@article {Zi_23,
    AUTHOR = {Ziegler, Volker},
     TITLE = {Sums of {F}ibonacci numbers that are perfect powers},
   JOURNAL = {Quaest. Math.},
  FJOURNAL = {Quaestiones Mathematicae. Journal of the South African
              Mathematical Society},
    VOLUME = {46},
      YEAR = {2023},
    NUMBER = {8},
     PAGES = {1717--1742},
      ISSN = {1607-3606,1727-933X},
   MRCLASS = {11D61 (11B39 11J86 11Y50)},
  MRNUMBER = {4623323},
MRREVIEWER = {Mahadi\ Ddamulira},
       DOI = {10.2989/16073606.2022.2109220},
       URL = {https://doi.org/10.2989/16073606.2022.2109220},
}

@article {BMS_06,
    AUTHOR = {Bugeaud, Yann and Mignotte, Maurice and Siksek, Samir},
     TITLE = {Classical and modular approaches to exponential {D}iophantine
              equations. {I}. {F}ibonacci and {L}ucas perfect powers},
   JOURNAL = {Ann. of Math. (2)},
  FJOURNAL = {Annals of Mathematics. Second Series},
    VOLUME = {163},
      YEAR = {2006},
    NUMBER = {3},
     PAGES = {969--1018},
      ISSN = {0003-486X,1939-8980},
   MRCLASS = {11D61 (11B39 11D59 11J86)},
  MRNUMBER = {2215137},
MRREVIEWER = {Yuri\ Bilu},
       DOI = {10.4007/annals.2006.163.969},
       URL = {https://doi.org/10.4007/annals.2006.163.969},
}

@article {KKLL_21,
    AUTHOR = {Kebli, Salima and Kihel, Omar and Larone, Jesse and Luca,
              Florian},
     TITLE = {On the nonnegative integer solutions to the equation {$F_ n\pm
              F_m = y^a$}},
   JOURNAL = {J. Number Theory},
  FJOURNAL = {Journal of Number Theory},
    VOLUME = {220},
      YEAR = {2021},
     PAGES = {107--127},
      ISSN = {0022-314X,1096-1658},
   MRCLASS = {11B39 (11D61 11J86)},
  MRNUMBER = {4177538},
MRREVIEWER = {Mahadi\ Ddamulira},
       DOI = {10.1016/j.jnt.2020.08.004},
       URL = {https://doi.org/10.1016/j.jnt.2020.08.004},
}

@article {LP18,
    AUTHOR = {Luca, Florian and Patel, Vandita},
     TITLE = {On perfect powers that are sums of two {F}ibonacci numbers},
   JOURNAL = {J. Number Theory},
  FJOURNAL = {Journal of Number Theory},
    VOLUME = {189},
      YEAR = {2018},
     PAGES = {90--96},
      ISSN = {0022-314X,1096-1658},
   MRCLASS = {11D61 (11B39)},
  MRNUMBER = {3788641},
MRREVIEWER = {L\'aszl\'o\ Szalay},
       DOI = {10.1016/j.jnt.2018.02.003},
       URL = {https://doi.org/10.1016/j.jnt.2018.02.003},
}

@article {VZ_24,
    AUTHOR = {Vukusic, Ingrid and Ziegler, Volker},
     TITLE = {On sums of two {F}ibonacci numbers that are powers of numbers
              with limited {H}amming weight},
   JOURNAL = {Quaest. Math.},
  FJOURNAL = {Quaestiones Mathematicae. Journal of the South African
              Mathematical Society},
    VOLUME = {47},
      YEAR = {2024},
    NUMBER = {4},
     PAGES = {851--869},
      ISSN = {1607-3606,1727-933X},
   MRCLASS = {11B39 (11D61 11J86)},
  MRNUMBER = {4739972},
       DOI = {10.2989/16073606.2023.2256477},
       URL = {https://doi.org/10.2989/16073606.2023.2256477},
}

@book {AS,
    AUTHOR = {Allouche, Jean-Paul and Shallit, Jeffrey},
     TITLE = {Automatic sequences},
      NOTE = {Theory, applications, generalizations},
 PUBLISHER = {Cambridge University Press, Cambridge},
      YEAR = {2003},
     PAGES = {xvi+571},
      ISBN = {0-521-82332-3},
   MRCLASS = {11B85 (11Z05 37A45 37B10 68Q45 68R15 94A45)},
  MRNUMBER = {1997038},
MRREVIEWER = {Val\'{e}rie\ Berth\'{e}},
       DOI = {10.1017/CBO9780511546563},
       URL = {https://doi.org/10.1017/CBO9780511546563},
}

@article {LeSh93,
    AUTHOR = {Lenstra, Jr., H. W. and Shallit, J. O.},
     TITLE = {Continued fractions and linear recurrences},
   JOURNAL = {Math. Comp.},
  FJOURNAL = {Mathematics of Computation},
    VOLUME = {61},
      YEAR = {1993},
    NUMBER = {203},
     PAGES = {351--354},
      ISSN = {0025-5718,1088-6842},
   MRCLASS = {11A55 (11B37)},
  MRNUMBER = {1192972},
MRREVIEWER = {Les\ Davison},
       DOI = {10.2307/2152958},
       URL = {https://doi.org/10.2307/2152958},
}

@article {BCM_13,
    AUTHOR = {Bugeaud, Yann and Cipu, Mihai and Mignotte, Maurice},
     TITLE = {On the representation of {F}ibonacci and {L}ucas numbers in an
              integer base},
   JOURNAL = {Ann. Math. Qu\'{e}.},
  FJOURNAL = {Annales Math\'{e}matiques du Qu\'{e}bec},
    VOLUME = {37},
      YEAR = {2013},
    NUMBER = {1},
     PAGES = {31--43}
}

@article {Os_1922,
    AUTHOR = {Ostrowski, Alexander},
     TITLE = {Bemerkungen zur {T}heorie der {D}iophantischen
              {A}pproximationen},
   JOURNAL = {Abh. Math. Sem. Univ. Hamburg},
  FJOURNAL = {Abhandlungen aus dem Mathematischen Seminar der
              Universit\"{a}t Hamburg},
    VOLUME = {1},
      YEAR = {1922},
    NUMBER = {1},
     PAGES = {77--98},
      ISSN = {0025-5858,1865-8784},
   MRCLASS = {99-04},
  MRNUMBER = {3069389},
       DOI = {10.1007/BF02940581},
       URL = {https://doi.org/10.1007/BF02940581},
}

@article {Mat1,
    AUTHOR = {Matveev, E. M.},
     TITLE = {An explicit lower bound for a homogeneous rational linear form
              in logarithms of algebraic numbers},
   JOURNAL = {Izv. Ross. Akad. Nauk Ser. Mat.},
  FJOURNAL = {Izvestiya Rossiiskoi Akademii Nauk. Seriya Matematicheskaya},
    VOLUME = {62},
      YEAR = {1998},
    NUMBER = {4},
     PAGES = {81--136},
      ISSN = {1607-0046,2587-5906},
   MRCLASS = {11J86},
  MRNUMBER = {1660150},
MRREVIEWER = {Yuri\ Bilu},
       DOI = {10.1070/im1998v062n04ABEH000190},
       URL = {https://doi.org/10.1070/im1998v062n04ABEH000190},
}

@article {Mat2,
    AUTHOR = {Matveev, E. M.},
     TITLE = {An explicit lower bound for a homogeneous rational linear form
              in logarithms of algebraic numbers. {II}},
   JOURNAL = {Izv. Ross. Akad. Nauk Ser. Mat.},
  FJOURNAL = {Izvestiya Rossiiskoi Akademii Nauk. Seriya Matematicheskaya},
    VOLUME = {64},
      YEAR = {2000},
    NUMBER = {6},
     PAGES = {125--180},
      ISSN = {1607-0046,2587-5906},
   MRCLASS = {11J86},
  MRNUMBER = {1817252},
MRREVIEWER = {Yuri\ Bilu},
       DOI = {10.1070/IM2000v064n06ABEH000314},
       URL = {https://doi.org/10.1070/IM2000v064n06ABEH000314},
}

@book {Wa_00,
    AUTHOR = {Waldschmidt, Michel},
     TITLE = {Diophantine approximation on linear algebraic groups},
    SERIES = {Grundlehren der mathematischen Wissenschaften [Fundamental
              Principles of Mathematical Sciences]},
    VOLUME = {326},
      NOTE = {Transcendence properties of the exponential function in
              several variables},
 PUBLISHER = {Springer-Verlag, Berlin},
      YEAR = {2000},
     PAGES = {xxiv+633},
      ISBN = {3-540-66785-7},
   MRCLASS = {11Jxx (20G99)},
  MRNUMBER = {1756786},
MRREVIEWER = {Yann\ Bugeaud},
       DOI = {10.1007/978-3-662-11569-5},
       URL = {https://doi.org/10.1007/978-3-662-11569-5},
}

@article {Zi_19,
    AUTHOR = {Ziegler, Volker},
     TITLE = {Effective results for linear equations in members of two
              recurrence sequences},
   JOURNAL = {Acta Arith.},
  FJOURNAL = {Acta Arithmetica},
    VOLUME = {190},
      YEAR = {2019},
    NUMBER = {2},
     PAGES = {139--169},
      ISSN = {0065-1036,1730-6264},
   MRCLASS = {11D61 (11A67 11B37 11B39)},
  MRNUMBER = {3984263},
MRREVIEWER = {P.\ Bundschuh},
       DOI = {10.4064/aa180427-13-11},
       URL = {https://doi.org/10.4064/aa180427-13-11},
}
\end{document}